\documentclass[a4, 12pt]{amsart}
\usepackage[mathscr]{eucal}
\usepackage{amssymb}
\usepackage{latexsym}
\usepackage{amsthm}
\usepackage{color}
\theoremstyle{plain}
\newtheorem{theorem}{Theorem}[section]
\newtheorem{problem}{Problem}[section]

\newtheorem{remark}{Remark}[section]
\newtheorem{lemma}{Lemma}[section]
\newtheorem{proposition}{Proposition}[section]

\makeatletter
\@addtoreset{equation}{section}

\title[Lagrangian translators and Lagrangian self-expanders]
{Classification of Lagrangian translators and Lagrangian self-expanders in $\mathbb{C}^{2}$}

\author [Z. Li and G. Wei]{Zhi Li and Guoxin Wei}

\address{Zhi Li \\  \newline \indent College of Mathematics and Information Science, Henan Normal University,
\newline \indent 453007, Xinxiang, Henan, China.}
\email{lizhihnsd@126.com}

\address{Guoxin Wei \\ \newline \indent School of Mathematical Sciences, South China Normal University,
\newline \indent 510631, Guangzhou,  China.}
\email{weiguoxin@tsinghua.org.cn}

\begin{document}
\maketitle

\begin{abstract}
In this paper, we obtain several classification results of $2$-dimensional complete Lagrangian translators and lagrangian self-expanders with constant squared norm $|\vec{H}|^{2}$ of the mean curvature vector in $\mathbb{C}^{2}$ by using a new Omori-Yau type maximum principle which was proved by Chen and Qiu \cite{CQ}. The same idea is also used to give a similar result of Lagrangian $\xi$-translators in $\mathbb{C}^{2}$.
\end{abstract}


\section{introduction}
\vskip2mm
\noindent

The mean curvature flow in the Euclidean space $\mathbb{R}^{n+p}$ is a one-parameter family
of immersions $x_{t}=x(\cdot, t): M^{n}\to \mathbb{R}^{n+p}$ with the corresponding image
$M_{t}=x_{t}(M^{n})$ such that
\begin{equation*}
\frac{d}{dt}x(p, t)=\vec{H}(p,t), \ \ x(p, 0)=x(p), \ \ p\in M^{n},
\end{equation*}
where $\vec{H}(p,t)$ is the mean curvature vector of $M_{t}$ at $p\in M^{n}$.

An $n$-dimensional smooth immersed submanifold $x: M^{n}\to \mathbb{R}^{n+p}$ is called a translating soliton (or, simply, translator) of the mean curvature flow if its mean curvature vector $\vec{H}$ satisfies the following equation
\begin{equation}\label{1.1-1}
\vec{H}+T^{\perp}=0,
\end{equation}
where $T^{\perp}$ denotes the normal part of nonzero constant vector $T$ in $\mathbb{R}^{n+p}$.
while $x(M^{n})$ is said to be a self-expander of the mean curvature if it holds that
\begin{equation}\label{1.1-2}
\vec{H}=x^{\perp},
\end{equation}
where $x^{\perp}$ is the orthogonal projection of the position vector $x$ in $\mathbb{R}^{n+p}$ to the normal bundle of $M^{n}$.
As is known to all, solutions of \eqref{1.1-1} correspond to translating solutions
 $\{M_{t}=M+tT, \ t\in \mathbb{R}\}$ of the mean
curvature flow, and are important in the singularity theory of the mean curvature flow since
they often occur as Type-II singularities. $M^{n}$ is a self-expander if and only if the family of homothetic hypersurfaces $\{x_{t}=\sqrt{2t}x, \ t>0\}$ is a mean curvature flow. The self-expanders appear as the singularity model of the mean curvature flow which exists for long time. Due to this, these two kinds of solitons are called self-similar solutions to the mean curvature flow of submanifolds in $\mathbb{R}^{n+p}$. They have been extensively studied for years and a number of interesting rigidity theorems and classification theorems have been obtained, including rigidity theorems and classification theorems. To see the details, readers are referred to, for translators, (\cite{AW}, \cite{CSS}, \cite{H1}, \cite{M}, \cite{Pyo}, \cite{Wang}, \cite{White}) and, for self-expanders, (\cite{AC}, \cite{CZ}, \cite{H}, \cite{Ish}, \cite{Smo}) etc.

It is known that Lagrangian submanifolds are a class of important submanifolds in the complex Euclidean space $\mathbb{C}^{n}$. Since the mean curvature flow keep invariant of the Lagrangian property, which
means that if the initial submanifold $x: M^{n}\to \mathbb{R}^{2n}$ is Lagrangian, then the mean curvature flow $x(\cdot, t): M^{n}\to \mathbb{R}^{2n}$
is also Lagrangian. Thus, the Lagrangian self-shrinkers, the Lagrangian translators and the Lagrangian self-expanders seem very interesting to study. To learn more about them, please refer to (\cite{Anc}, \cite{Cas}, \cite{GSSZ}, \cite{IJS}, \cite{JLT}, \cite{LW}, \cite{LW1}, \cite{Nak}, \cite{Nev}, \cite{NT}) and the references therein.

From view points of submanifolds theory, it is also very natural to study rigidity
and classification theorems for the Lagrangian self-shrinkers, the Lagrangian translators and the Lagrangian self-expanders. In this direction, Castro and Lerma (\cite{Cas1})provided several rigidity results for the Clifford torus in the class of compact self-shrinkers for Lagrangian mean curvature flow. In 2017, Li and Wang (\cite{LW3}) prove a rigidity theorem which improves a previous theorem
by Castro and Lerma (\cite{Cas1}). Cheng, Hori and Wei (\cite{CHW}) established an interesting classification theorem for complete Lagrangian self-shrinkers with
constant squared norm of the second fundamental form in $\mathbb{C}^{2}$.
Later, Li et al.(\cite{LLQ}) and the author of the present paper (\cite{LW4}) respectively studied the classification of complete the complete Lagrangian translators and Lagrangian self-expanders in $\mathbb{C}^{2}$.
Recently, under the condition that the squared norm of the second fundamental form being bounded from above, the author of the present paper (\cite{LWW}) studied the rigidity problem for $2$-dimensional complete Lagrangian self-shrinkers with constant squared norm $|\vec{H}|^{2}$ of the mean curvature vector in $\mathbb{C}^{2}$. It is natural to ask the following problems:

\begin{problem}\label{problem 1.1}
To classify $2$-dimensional complete Lagrangian translators or Lagrangian self-expanders with constant squared norm $|\vec{H}|^{2}$ of the mean curvature vector in $\mathbb{C}^{2}$
if the squared norm of the second fundamental form is bounded from above.
\end{problem}

It is our motivation to solve the above problem. In this paper, we solve the Problem \ref{problem 1.1}.

\begin{theorem}\label{theorem 1.1}
Let $x: M^{2}\to \mathbb C^{2}$ be a
$2$-dimensional complete Lagrangian translator with constant squared norm $|\vec{H}|^{2}$ of the mean curvature vector in $\mathbb{C}^{2}$. If the squared norm of the second fundamental form is bounded from above, then $x(M^{2})$ is a plane
$\mathbb {R}^{2}$.
\end{theorem}

Assume that $T$ is a nonzero constant in $\mathbb{R}^{n+p}$, then $M^{n}$ is called a $\xi$-translator if $\xi=\vec{H}+T^{\perp}$
is parallel in the normal bundle. It is easy to see that when $\xi=0$, $M^{n}$ is a translating soliton (translator). Thus, by an application of Theorem \ref{theorem 1.1} and a theorem
of Hoffman (Theorem 4.1, \cite{Hoff}), we can easily obtain the following more general result.

\begin{theorem}\label{theorem 1.2}
Let $x: M^{2}\to \mathbb C^{2}$ be a $2$-dimensional complete Lagrangian $\xi$-translator with constant squared norm $|\vec{H}|^{2}$ of the mean curvature vector in $\mathbb{C}^{2}$. If the squared norm of the second fundamental form is bounded from above, then $x(M^{2})$ is a plane $\mathbb {R}^{2}$ or the circular cylinder $\mathbb{S}^{1}(r)\times \mathbb{R}^{1}$ for some $r>0$.
\end{theorem}

\begin{remark}
Note that any sphere $\mathbb{S}^{2}(r)$, $r>0$ or any Clifford torus $\mathbb{S}^{1}(r_{1})\times \mathbb{S}^{1}(r_{2})$,
$r_{1}, \ r_{2}>0$, can not be an immersed $\xi$-translator.
\end{remark}

By using a similar proof method of the Theorem \ref{theorem 1.1}, we also
obtain
\begin{theorem}\label{theorem 1.3}
Let $x: M^{2}\to \mathbb C^{2}$ be a
$2$-dimensional complete Lagrangian self-expander with constant squared norm of the mean curvature vector in $\mathbb{C}^{2}$. If the squared norm of the second fundamental form is bounded from above, then $x(M^{2})$ is a plane $\mathbb {R}^{2}$ through the origin.
\end{theorem}

\vskip5mm
\section {Preliminaries}
\vskip2mm

\noindent

Let $x: M^{2} \rightarrow\mathbb C^{2}$ be an
$2$-dimensional Lagrangian surface of $\mathbb C^{2}$. Denote
by $J$ the canonical complex structure on $\mathbb C^{2}$.
We choose orthonormal tangent vector fields $\{e_{1}, e_{2}\}$ and $\{e_{1^{\ast}}, e_{2^{\ast}}\}$ are normal vector fields given by
$$e_{1^{\ast}}=J e_{1}, \ e_{2^{\ast}}=Je_{2}.$$
Then
$$\{e_{1}, e_{2}, e_{1^{\ast}}, e_{2^{\ast}}\}$$
is called an adapted Lagrangian frame field. The dual frame fields of $\{e_{1}, e_{2}\}$ are $\{\omega_{1}, \omega_{2}\}$, the Levi-Civita connection forms and normal
connection forms are $\omega_{ij}$ and $\omega_{i^{\ast}j^{\ast}}$ , respectively.

Since $x: M^{2} \rightarrow\mathbb C^{2}$ is a Lagrangian surface (see \cite{LV}, \cite{LW2}), we have
\begin{equation}\label{2.1-1}
h^{p^{\ast}}_{ij}=h^{p^{\ast}}_{ji}=h^{i^{\ast}}_{pj}, \ \ i,j,p=1,2.
\end{equation}
The second fundamental form $h$ and the mean curvature $\vec{H}$ of $x$ are respectively
defined by $$h=\sum_{ijp}h^{p^{\ast}}_{ij}\omega_{i}\otimes\omega_{j}\otimes e_{p^{\ast}},\ \ \vec{H}=\sum_{p}H^{p^{\ast}}e_{p^{\ast}}=\sum_{i,p}h^{p^{\ast}}_{ii}e_{p^{\ast}}.$$
Let $S=\sum_{i,j,p}(h^{p^{\ast}}_{ij})^{2}$ be the squared norm of the second
fundamental form and $H=|\vec{H}|$ denote the mean curvature of $x$. If we denote the components of curvature tensors of the Levi-Civita connection forms $\omega_{ij}$ and normal connection forms $\omega_{i^{\ast}j^{\ast}}$ by
$R_{ijkl}$ and $R_{i^{\ast}j^{\ast}kl}$, respectively, then the equations of Gauss, Codazzi and Ricci are given by
\begin{equation}\label{2.1-2}
R_{ijkl}=\sum_{p}(h^{p^{\ast}}_{ik}h^{p^{\ast}}_{jl}-h^{p^{\ast}}_{il}h^{p^{\ast}}_{jk}),
\end{equation}
\begin{equation}\label{2.1-3}
R_{ik}=\sum_{p}H^{p^{\ast}}h^{p^{\ast}}_{ik}-\sum_{j,p}h^{p^{\ast}}_{ij}h^{p^{\ast}}_{jk},
\end{equation}
\begin{equation}\label{2.1-4}
h^{p^{\ast}}_{ijk}=h^{p^{\ast}}_{ikj},
\end{equation}
\begin{equation}\label{2.1-5}
R_{p^{\ast}q^{\ast}kl}=\sum_{i}(h^{p^{\ast}}_{ik}h^{q^{\ast}}_{il}
-h^{p^{\ast}}_{il}h^{p^{\ast}}_{ik}),
\end{equation}
\begin{equation}\label{2.1-6}
R=H^{2}-S.
\end{equation}

From \eqref{2.1-1} and \eqref{2.1-4}, we easily know that the components $h^{p^{\ast}}_{ijk}$ is totally symmetric for $i, j, k,l$. In particular,
\begin{equation}\label{2.1-7}
h^{p^{\ast}}_{ijk}=h^{p^{\ast}}_{kji}=h^{i^{\ast}}_{pjk}, \ \ i, j, k ,p=1, 2.
\end{equation}
By making use of \eqref{2.1-1}, \eqref{2.1-2} and \eqref{2.1-5}, we obtain
\begin{equation}\label{2.1-8}
R_{ijkl}=K(\delta_{ik}\delta_{jl}-\delta_{il}\delta_{jk})=R_{i^{\ast}j^{\ast}kl}, \ \ K=\frac{1}{2}(H^{2}-S),
\end{equation}
where $K$ is the Gaussian curvature of $x$.

\noindent By defining
\begin{equation*}
\sum_{l}h^{p^{\ast}}_{ijkl}\omega_{l}=dh^{p^{\ast}}_{ijk}+\sum_{l}h^{p^{\ast}}_{ljk}\omega_{li}
+\sum_{l}h^{p^{\ast}}_{ilk}\omega_{lj}+\sum_{l} h^{p^{\ast}}_{ijl}\omega_{lk}+\sum_{q} h^{q^{\ast}}_{ijk}\omega_{q^{\ast}p^{\ast}},
\end{equation*}
we have the following Ricci identities
\begin{equation}\label{2.1-9}
h^{p^{\ast}}_{ijkl}-h^{p^{\ast}}_{ijlk}=\sum_{m}
h^{p^{\ast}}_{mj}R_{mikl}+\sum_{m} h^{p^{\ast}}_{im}R_{mjkl}+\sum_{m} h^{m^{\ast}}_{ij}R_{m^{\ast}p^{\ast}kl}.
\end{equation}

Let $V$ be a tangent $C^{1}$-vector field on $M^{n}$ and denote by $Ric_{V} := Ric-\frac{1}{2}L_{V}g$ the
Bakry-Emery Ricci tensor with $L_{V}$ to be the Lie derivative along the vector field $V$. Define a
differential operator
\begin{equation*}
\mathcal{L}_{V}f=\Delta f+\langle V,\nabla f\rangle,
\end{equation*}
where $\Delta$ and $\nabla$ denote the Laplacian and the gradient
operator, respectively. The following maximum principle
of Omori-Yau type concerning the operator $\mathcal{L}$ will
be used in this paper, which was proved by Chen and Qiu \cite{CQ}.

\begin{lemma}\label{lemma 2.1}
Let $(M^{n}, g)$ be a complete Riemannian manifold, and $V$ is a $C^{1}$ vector field on $M^{n}$. If the Bakry-Emery Ricci tensor $Ric_{V}$ is bounded from below, then for any $f\in C^{2}(M^{n})$ bounded from above, there exists a sequence $\{p_{t}\} \subset M^{n}$, such that
\begin{equation*}
\lim_{m\rightarrow\infty} f(p_{t})=\sup f,\quad
\lim_{m\rightarrow\infty} |\nabla f|(p_{t})=0,\quad
\lim_{m\rightarrow\infty}\mathcal{L}_{V}f(p_{t})\leq 0.
\end{equation*}
\end{lemma}

For the mean curvature vector field $\vec{H}=\sum_{p}H^{p^{\ast}}e_{p^{\ast}}$, we define
\begin{equation}\label{2.1-10}
|\nabla^{\perp}\vec{H}|^{2}=\sum_{i,p}(H^{p^{\ast}}_{,i})^{2}, \ \ \Delta^{\perp}H^{p^{\ast}}=\sum_{i}H^{p^{\ast}}_{,ii}.
\end{equation}
By the definition \eqref{1.1-1} and \eqref{1.1-2} of the Lagrangian translator and lagrangian self-expander respectively, it is sufficient
to give several basic differential formulas.
\begin{equation}\label{2.1-11}
H^{p^{\ast}}_{,i}
=\sum_{k}h^{p^{\ast}}_{ik}\langle T, e_{k}\rangle, \ \
H^{p^{\ast}}_{,ij}
=\sum_{k}h^{p^{\ast}}_{ijk}\langle T, e_{k}\rangle-\sum_{k,q}h^{p^{\ast}}_{ik}h^{q^{\ast}}_{kj}H^{q^{\ast}}
\end{equation}
and
\begin{equation}\label{2.1-12}
H^{p^{\ast}}_{,i}
=-\sum_{k}h^{p^{\ast}}_{ik}\langle x, e_{k}\rangle, \ \
H^{p^{\ast}}_{,ij}
=-\sum_{k}h^{p^{\ast}}_{ijk}\langle x, e_{k}\rangle-h^{p^{\ast}}_{ij}-\sum_{k,q}h^{p^{\ast}}_{ik}h^{q^{\ast}}_{kj}H^{q^{\ast}}.
\end{equation}

\noindent
If we choose $V=-T^{\top}$ and $x^{\top}$ respectively, using the above formulas and the Ricci identities, we can get the following Lemmas (see \cite{LLQ} and \cite{LW4}).
\begin{lemma}\label{lemma 2.2}
Let $x:M^{2}\rightarrow \mathbb{C}^{2}$ be an $2$-dimensional complete lagrangian translator. We have
\begin{equation}\label{2.1-13}
\frac{1}{2}\mathcal{L}_{-T^{\top}} H^{2}=\sum_{i,p}(H^{p^{\ast}}_{,i})^{2}-\sum_{i,j,p,q}H^{p^{\ast}}h^{p^{\ast}}_{ij}H^{q^{\ast}}h^{q^{\ast}}_{ij}
\end{equation}
and
\begin{equation}\label{2.1-14}
\frac{1}{2}\mathcal{L}_{-T^{\top}}S
=\sum_{i,j,k}(h^{p^{\ast}}_{ijk})^{2}-\frac{1}{2}(H^{2}-S)(H^{2}-3S)
-\sum_{i,j,p,q}H^{p^{\ast}}h^{p^{\ast}}_{ij}H^{q^{\ast}}h^{q^{\ast}}_{ij}.
\end{equation}
\end{lemma}

\begin{lemma}\label{lemma 2.3}
Let $x:M^{2}\rightarrow \mathbb{C}^{2}$ be an $2$-dimensional complete lagrangian self-expander. We have
\begin{equation}\label{2.1-15}
\frac{1}{2}\mathcal{L}_{x^{\top}} H^{2}=\sum_{i,p}(H^{p^{\ast}}_{,i})^{2}-H^{2}-\sum_{i,j,p,q}H^{p^{\ast}}h^{p^{\ast}}_{ij}H^{q^{\ast}}h^{q^{\ast}}_{ij}
\end{equation}
and
\begin{equation}\label{2.1-16}
\frac{1}{2}\mathcal{L}_{x^{\top}}S
=\sum_{i,j,k}(h^{p^{\ast}}_{ijk})^{2}-S(\frac{3}{2}S+1)+2H^{2}S-\frac{1}{2}H^{4}
-\sum_{i,j,p,q}H^{p^{\ast}}h^{p^{\ast}}_{ij}H^{q^{\ast}}h^{q^{\ast}}_{ij}.
\end{equation}
\end{lemma}

In order to use the maximum principle of Omori-Yau type (Lemma \ref{lemma 2.1}), we need the following conclusions. The specific proof approach for the following conclusions is similar to \cite{LLQ} and \cite{LW4}. For completeness, we will present a short proof of it.
\begin{lemma}\label{lemma 2.4}
For a complete translator $x:M^{n}\rightarrow \mathbb{C}^{n}$ with the squared norm $S$ of the second fundamental form being bounded from above, the Bakry-Emery Ricci tensor $Ric_{V}$ is bounded from below, where $V=-T^{\top}$.
\end{lemma}
\begin{proof}
For any unit vector $e\in T M^{n}$, we can
choose a local tangent orthonormal frame field $\{{e_{i}}\}^{n}_{i=1}$ such that $e=e_{i}$. By the definition of translator and a simple computation, we have
\begin{equation*}
\begin{aligned}
-\frac{1}{2}L_{-T^{\top}}g(e_{i},e_{i})
=&\frac{1}{2}T^{\top}(g(e_{i},e_{i}))-g([T^{\top},e_{i}],e_{i})\\
=&g(\nabla_{e_{i}}(T-T^{\bot}),e_{i})
=\sum_{p}H^{p^{\ast}}g(\nabla_{e_{i}}e_{p^{\ast}},e_{i})\\
=&-\sum_{p}H^{p^{\ast}}h^{p^{\ast}}_{ii}.
\end{aligned}
\end{equation*}

Then \eqref{2.1-3} yields
\begin{equation*}
Ric_{-T^{\top}}(e_{i},e_{i})
=Ric(e_{i},e_{i})-\frac{1}{2}L_{-T^{\top}}g(e_{i},e_{i})
=-\sum_{j,p}(h^{p^{\ast}}_{ij})^{2}\geq-S.
\end{equation*}
It is natural to draw that Bakry-Emery Ricci tensor $Ric_{-T^{\top}}$ is bounded from below since $S$ is bounded from above.
\end{proof}

Using a similar discussion method, we can also draw the following conclusion.
\begin{lemma}\label{lemma 2.5}
For a complete self-expander $x:M^{n}\rightarrow \mathbb{C}^{n}$ with the squared norm $S$ of the second fundamental form being bounded from above, the Bakry-Emery Ricci tensor $Ric_{V}$ is bounded from below, where $V=x^{\top}$.
\end{lemma}

 \vskip10mm
\section{Proof of Main Theorem}

\vskip2mm
To draw the conclusion of Theorem \ref{theorem 1.1}, we need the following proposition.

\begin{proposition}\label{proposition 3.1}
Let $x:M^{2}\rightarrow \mathbb{C}^{2}$
be a Lagrangian translator with constant squared norm $|\vec{H}|^{2}$ of the mean curvature vector. If the squared norm $S$ of the second fundamental form is bounded from above, then $|\vec{H}|^{2}\equiv0$.
\end{proposition}
\begin{proof}

If we had $|\vec{H}|^{2}\neq0$, we choose a local frame field $\{e_{1}, e_{2}\}$
such that
$$ \vec{H}=H^{1^{\ast}}e_{1^{\ast}}, \ \ H^{1^{\ast}}=|\vec{H}|=H, \ \ H^{2^{\ast}}=h^{2^{\ast}}_{11}+h^{2^{\ast}}_{22}=0.
$$
Then,
$$S=(h^{1^{\ast}}_{11})^{2}+3(h^{1^{\ast}}_{22})^{2}+4(h^{2^{\ast}}_{11})^{2}, \ \ H^{2}=(h^{1^{\ast}}_{11}+h^{1^{\ast}}_{22})^{2}\leq \frac{4}{3}\Big((h^{1^{\ast}}_{11})^{2}+3(h^{1^{\ast}}_{22})^{2}\Big)\leq \frac{4}{3}S$$
and the equality of the above inequality holds if and only if
$$h^{1^{\ast}}_{11}=3h^{1^{\ast}}_{22}, \ \ h^{2^{\ast}}_{11}=0.$$
Since $|\vec{H}|^{2}$ is constant, the lemma \ref{lemma 2.2} implies
\begin{equation}\label{3.1-1}
\sum_{i,p}(H^{p^{\ast}}_{,i})^{2}=\sum_{i,j,p,q}H^{p^{\ast}}h^{p^{\ast}}_{ij}H^{q^{\ast}}h^{q^{\ast}}_{ij}
\end{equation}
and
\begin{equation}\label{3.1-2}
\frac{1}{2}\mathcal{L}_{-T^{\top}}S=\sum_{i,j,k}(h^{p^{\ast}}_{ijk})^{2}
-\sum_{i,p}(H^{p^{\ast}}_{,i})^{2}-\frac{1}{2}(H^{2}-S)(H^{2}-3S).
\end{equation}
Since $S$ is bounded from above, we know that the Bakry-Emery Ricci curvature of $x:M^{2}\rightarrow \mathbb{C}^{2}$ is bounded from below from the lemma \ref{lemma 2.4}. By applying the maximum principle
of Omori-Yau type concerning the operator $\mathcal{L}_{-T^{\top}}$ to the function $-S$, there exists a sequence $\{p_{t}\} \subset M^{2}$ such that
\begin{equation*}
\lim_{t\rightarrow\infty} S(p_{t})=\inf S, \ \
\lim_{t\rightarrow\infty} |\nabla S|(p_{t})=0, \ \
\lim_{t\rightarrow\infty}\mathcal{L}_{-T^{\top}} S(p_{t})\geq 0.
\end{equation*}
And because $S$ is bounded from above, we know that
$\{h^{p^{\ast}}_{ij}(p_{t})\}$ are bounded sequences for $ i, j, p = 1,2$.
Hence we can assume
\begin{equation*}
\lim_{t\rightarrow\infty} S(p_{t})=\inf S=\bar S, \ \ \lim_{t\rightarrow\infty}h^{p^{\ast}}_{ij}(p_{t})=\bar h^{p^{\ast}}_{ij},
 \ \ i, j, p=1, 2.
\end{equation*}
Without loss of the generality, we can
assume $h^{p^{\ast}}_{ij}(p_{t})\neq0$ for $i, j, p=1, 2, 3$. Unless otherwise specified, the following equations are considered at point $p_{t}\in M^{2}$.

If $\inf S=0$, we draw that $|\vec{H}|^{2}\equiv0$ since $|\vec{H}|^{2}\leq\frac{4}{3}S$.
Next, we will only consider $\inf S>0$. In fact, this situation does not exist.

\noindent
Since $|\nabla H^{2}|=0$ and $|\nabla H^{2}|^{2}=4\sum_{k}(\sum_{p}H^{p^{\ast}}H^{p^{\ast}}_{,k})^{2}$,
we can see that
\begin{equation}\label{3.1-3}
H^{1^{\ast}}_{,k}=H^{k^{\ast}}_{,1}=0, \ \ h^{1^{\ast}}_{11k}+h^{1^{\ast}}_{22k}=0, \ \ k=1, 2.
\end{equation}
It follows from the first formula of \eqref{2.1-11} and $h^{2^{\ast}}_{11}+h^{2^{\ast}}_{22}=0$ that
\begin{equation}\label{3.1-4}
H^{1^{\ast}}_{,1}=h^{1^{\ast}}_{11}\langle T, e_{1} \rangle+h^{2^{\ast}}_{11}\langle T, e_{2} \rangle, \ \
H^{1^{\ast}}_{,2}=h^{2^{\ast}}_{11}\langle T, e_{1} \rangle+h^{1^{\ast}}_{22}\langle T, e_{2} \rangle
\end{equation}
and
\begin{equation}\label{3.1-5}
H^{2^{\ast}}_{,1}=h^{2^{\ast}}_{11}\langle T, e_{1} \rangle+h^{1^{\ast}}_{22}\langle T, e_{2} \rangle, \ \
H^{2^{\ast}}_{,2}=h^{1^{\ast}}_{22}\langle T, e_{1} \rangle-h^{2^{\ast}}_{11}\langle T, e_{2} \rangle.
\end{equation}
Choosing $\nabla_{k}S=2a_{k}$ for $k=1,2$, \eqref{3.1-3} and $\lim_{t\rightarrow\infty} |\nabla S|(p_{t})=0$ imply that
\begin{equation}\label{3.1-6}
(h^{1^{\ast}}_{11}-3h^{1^{\ast}}_{22}) h^{1^{\ast}}_{11k}+3h^{2^{\ast}}_{11}h^{2^{\ast}}_{11k}-h^{2^{\ast}}_{11}h^{2^{\ast}}_{22k}=a_{k}, \ \ \lim_{t\rightarrow\infty}a_{k}(p_{t})=0, \ \ k=1, 2.
\end{equation}
Combining \eqref{3.1-3} and \eqref{3.1-6}, we infer
\begin{equation}\label{3.1-7}
\begin{aligned}
&\big((h^{1^{\ast}}_{11}-3h^{1^{\ast}}_{22})^{2}+12(h^{2^{\ast}}_{11})^{2}\big)h^{1^{\ast}}_{111}
+4(h^{2^{\ast}}_{11})^{2} h^{2^{\ast}}_{222}=b_{1}, \\
&\big((h^{1^{\ast}}_{11}-3h^{1^{\ast}}_{22})^{2}+12(h^{2^{\ast}}_{11})^{2}\big) h^{2^{\ast}}_{111}-h^{2^{\ast}}_{11}(h^{1^{\ast}}_{11}-3h^{1^{\ast}}_{22})h^{2^{\ast}}_{222}
=b_{2}.
\end{aligned}
\end{equation}
where $b_{1}=(h^{1^{\ast}}_{11}-3h^{1^{\ast}}_{22})a_{1}-4h^{2^{\ast}}_{11} a_{2}$,
$b_{2}=3h^{2^{\ast}}_{11} a_{1}+(h^{1^{\ast}}_{11}-3h^{1^{\ast}}_{22})a_{2}$ and
$\lim_{t\rightarrow\infty}b_{k}(p_{t})=0$ for $k=1,2$.

With all these preparations, we will use the proof by contradiction to prove that $$\bar h^{2^{\ast}}_{11}=0.$$
Now assume that $\bar h^{2^{\ast}}_{11}\neq0$. By
\eqref{3.1-3} and \eqref{3.1-4}, we get that
\begin{equation}\label{3.1-8}
\big(\bar h^{1^{\ast}}_{11}\bar h^{1^{\ast}}_{22}-(\bar h^{2^{\ast}}_{11})^{2}\big)\lim_{t\rightarrow\infty} \langle T, e_{1} \rangle(p_{t})=0, \ \
\big(\bar h^{1^{\ast}}_{11}\bar h^{1^{\ast}}_{22}-(\bar h^{2^{\ast}}_{11})^{2}\big)\lim_{t\rightarrow\infty} \langle T, e_{2} \rangle(p_{t})=0.
\end{equation}
If $\bar h^{1^{\ast}}_{11}\bar h^{1^{\ast}}_{22}
-(\bar h^{2^{\ast}}_{11})^{2}\neq 0$,
\eqref{3.1-8} yields
\begin{equation*}
\lim_{t\rightarrow\infty} \langle T, e_{1} \rangle(p_{t})=\lim_{t\rightarrow\infty} \langle T, e_{2} \rangle(p_{t})=0.
\end{equation*}
Thus, using \eqref{3.1-3} and \eqref{3.1-5}, we obtain that
\begin{equation}\label{3.1-9}
\lim_{t\rightarrow\infty}\big(\sum_{i,p}(H^{p^{\ast}}_{,i})^{2}\big)(p_{t})
=\lim_{t\rightarrow\infty}(H^{2^{\ast}}_{,2})^{2}(p_{t})=0.
\end{equation}
It is straightforward to see from \eqref{3.1-1} and \eqref{3.1-9} that
\begin{equation*}
\lim_{t\rightarrow\infty}\big(\sum_{i,j,p,q}H^{p^{\ast}}h^{p^{\ast}}_{ij}H^{q^{\ast}}h^{q^{\ast}}_{ij}\big)
(p_{t})=H^{2}\big((\bar h^{1^{\ast}}_{11})^{2}+(\bar h^{1^{\ast}}_{22})^{2}+2(\bar h^{2^{\ast}}_{11})^{2}\big)=0.
\end{equation*}
It contradicts the hypothesis.

\noindent
If $\bar h^{1^{\ast}}_{11}\bar h^{1^{\ast}}_{22}-(\bar h^{2^{\ast}}_{11})^{2}=0$,
it is obvious to draw $\bar h^{1^{\ast}}_{11}+3\bar h^{1^{\ast}}_{22}\neq0$, otherwise we would have $\bar h^{2^{\ast}}_{11}=0$. It contradicts the hypothesis.
Thus, \eqref{3.1-7} yields
\begin{equation}\label{3.1-10}
h^{1^{\ast}}_{111}=-\frac{4(h^{2^{\ast}}_{11})^{2}h^{2^{\ast}}_{222}-b_{1}}{(h^{1^{\ast}}_{11}-3h^{1^{\ast}}_{22})^{2}
+12(h^{2^{\ast}}_{11})^{2}}, \ \
h^{2^{\ast}}_{111}=\frac{h^{2^{\ast}}_{11}(h^{1^{\ast}}_{11}-3h^{1^{\ast}}_{22})h^{2^{\ast}}_{222}+b_{2}}{(h^{1^{\ast}}_{11}-3h^{1^{\ast}}_{22})^{2}
+12(h^{2^{\ast}}_{11})^{2}}.
\end{equation}
By \eqref{3.1-3} and \eqref{3.1-10}, a simple computation shows
\begin{equation*}
\sum_{i,p}(H^{p^{\ast}}_{,i})^{2}=(H^{2^{\ast}}_{,2})^{2}
=\Big(\frac{\big((h^{1^{\ast}}_{11}-3h^{1^{\ast}}_{22})^{2}
+16(h^{2^{\ast}}_{11})^{2}\big)h^{2^{\ast}}_{222}-b_{1}}{(h^{1^{\ast}}_{11}-3h^{1^{\ast}}_{22})^{2}
+12(h^{2^{\ast}}_{11})^{2}}\Big)^{2}
\end{equation*}
and
\begin{equation*}
\begin{aligned}
\sum_{i,j,k}(h^{p^{\ast}}_{ijk})^{2}
=&7(h^{1^{\ast}}_{111})^{2}
+8(h^{2^{\ast}}_{111})^{2}+(h^{2^{\ast}}_{222})^{2} \\
=&\frac{7\big(4(h^{2^{\ast}}_{11})^{2}h^{2^{\ast}}_{222}-b_{1}\big)^{2}
+8\big(h^{2^{\ast}}_{11}(h^{1^{\ast}}_{11}-3h^{1^{\ast}}_{22})h^{2^{\ast}}_{222}+b_{2}\big)^{2}}{\big((h^{1^{\ast}}_{11}-3h^{1^{\ast}}_{22})^{2}
+12(h^{2^{\ast}}_{11})^{2}\big)^{2}}+(h^{2^{\ast}}_{222})^{2}.
\end{aligned}
\end{equation*}
Then
\begin{equation*}
\begin{aligned}
\lim_{t\rightarrow\infty}\big(\sum_{i,p}(H^{p^{\ast}}_{,i})^{2}\big)(p_{t})
=&\frac{\big((\bar h^{1^{\ast}}_{11})^{2}+9(\bar h^{1^{\ast}}_{22})^{2}+10(\bar h^{2^{\ast}}_{11})^{2}\big)^{2}}{(\bar h^{1^{\ast}}_{11}+3\bar h^{1^{\ast}}_{22})^{4}}\lim_{t\rightarrow\infty}(h^{2^{\ast}}_{222})^{2}(p_{t}), \\
\lim_{t\rightarrow\infty}\big(\sum_{i,j,k}(h^{p^{\ast}}_{ijk})^{2}\big)(p_{t})
=&\frac{\big((\bar h^{1^{\ast}}_{11})^{2}+9(\bar h^{1^{\ast}}_{22})^{2}+10(\bar h^{2^{\ast}}_{11})^{2}\big)^{2}}{(\bar h^{1^{\ast}}_{11}+3\bar h^{1^{\ast}}_{22})^{4}}\lim_{t\rightarrow\infty}(h^{2^{\ast}}_{222})^{2}(p_{t}).
\end{aligned}
\end{equation*}
Hence from \eqref{3.1-2} one sees
\begin{equation*}
(H^{2}-\bar S)(H^{2}-3\bar S)\leq0
\end{equation*}
It is a contradiction since
$H^{2}=(\bar h^{1^{\ast}}_{11})^{2}+(\bar h^{1^{\ast}}_{22})^{2}+2(\bar h^{2^{\ast}}_{11})^{2}$ and $H^{2}<\bar S$.

Next, we will use $\bar h^{2^{\ast}}_{11}=0$ to complete the proof of proposition \ref{proposition 3.1}.

Since $\bar h^{2^{\ast}}_{11}=0$ and $|\lim_{t\rightarrow\infty} \langle T, e_{k} \rangle(p_{t})|<\infty$ for $k=1,2$, \eqref{3.1-3}, \eqref{3.1-4} and \eqref{3.1-5} show that
\begin{equation}\label{3.1-11}
\begin{aligned}
&\bar h^{1^{\ast}}_{11}\lim_{t\rightarrow\infty} \langle T, e_{1} \rangle(p_{t})=\bar h^{1^{\ast}}_{22}\lim_{t\rightarrow\infty} \langle T, e_{2} \rangle(p_{t})=0, \\
&\lim_{t\rightarrow\infty}H^{2^{\ast}}_{,2}(p_{t})=\bar h^{1^{\ast}}_{22}\lim_{t\rightarrow\infty} \langle T, e_{1} \rangle(p_{t}).
\end{aligned}
\end{equation}
If $\lim_{t\rightarrow\infty} \langle T, e_{1} \rangle(p_{t})=0$, \eqref{3.1-11} shows
$$\lim_{t\rightarrow\infty}\big(\sum_{i,p}(H^{p^{\ast}}_{,i})^{2}\big)(p_{t})
=\lim_{t\rightarrow\infty}H^{2^{\ast}}_{,2}(p_{t})=0.$$
Thus,
\begin{equation*} \lim_{t\rightarrow\infty}\big(\sum_{i,j,p,q}H^{p^{\ast}}h^{p^{\ast}}_{ij}H^{q^{\ast}}h^{q^{\ast}}_{ij}
\big)(p_{t})=H^{2}\big((\bar h^{1^{\ast}}_{11})^{2}+(\bar h^{1^{\ast}}_{22})^{2}\big)=0.
\end{equation*}
It contradicts the hypothesis.

\noindent
If $\lim_{t\rightarrow\infty} \langle T, e_{1} \rangle(p_{t})\neq0$, \eqref{3.1-11} yields
\begin{equation}\label{3.1-12}
\bar h^{1^{\ast}}_{11}=0, \ \ \bar h^{1^{\ast}}_{22}=H, \ \ \bar S=3H^{2}, \ \
\lim_{t\rightarrow\infty} \langle T, e_{2} \rangle(p_{t})=0.
\end{equation}
It is straightforward to calculate from $|\vec{H}|^{2}=constant$ and \eqref{3.1-1} that

\begin{equation*}
\sum_{p}H^{p^{\ast}}_{,i}H^{p^{\ast}}_{,j}
+\sum_{p}H^{p^{\ast}}H^{p^{\ast}}_{,ij}=0, \ \ i,j=1,2
\end{equation*}
and
\begin{equation*}
\sum_{i,p}H^{p^{\ast}}_{,i}H^{p^{\ast}}_{,ik}
=\sum_{i,j}(\sum_{p}H^{p^{\ast}}h^{p^{\ast}}_{ij})(\sum_{p}H^{p^{\ast}}_{,k}h^{p^{\ast}}_{ij}
+\sum_{p}H^{p^{\ast}}h^{p^{\ast}}_{ijk}), \ \ k=1,2.
\end{equation*}
Thus, choosing $i=j=1$ and $k=1$ in the above equations, we know
\begin{equation}\label{3.1-13}
H^{1^{\ast}}_{,11}=0, \ \
H^{2^{\ast}}_{,2}H^{2^{\ast}}_{,21}
=H^{2}\sum_{ij}h^{1^{\ast}}_{ij}h^{1^{\ast}}_{ij1}
=H^{2}\big((h^{1^{\ast}}_{11}-h^{1^{\ast}}_{22})h^{1^{\ast}}_{111}+2 h^{2^{\ast}}_{11}h^{2^{\ast}}_{111}\big).
\end{equation}
Choosing $k=1$ in \eqref{3.1-6}, we have
$$(h^{1^{\ast}}_{11}-3 h^{1^{\ast}}_{22})h^{1^{\ast}}_{111}+4 h^{2^{\ast}}_{11}h^{2^{\ast}}_{111}=a_{1}.$$
Hence by \eqref{3.1-13} and $H=h^{1^{\ast}}_{11}+h^{1^{\ast}}_{22}$, we can write
\begin{equation}\label{3.1-14}
H^{2^{\ast}}_{,2}H^{2^{\ast}}_{,21}
=\frac{1}{2}H^{2}(Hh^{1^{\ast}}_{111}+a_{1}).
\end{equation}
Besides, \eqref{2.1-11} yields that
\begin{equation*}
H^{1^{\ast}}_{,11}
=\sum_{k}h^{1^{\ast}}_{11k}\langle T, e_{k}\rangle-H\sum_{k}(h^{1^{\ast}}_{1k})^{2}, \ \
H^{2^{\ast}}_{,21}
=\sum_{k}h^{2^{\ast}}_{21k}\langle T, e_{k}\rangle-H\sum_{k}h^{1^{\ast}}_{1k}h^{2^{\ast}}_{2k}.
\end{equation*}
Then by $h^{1^{\ast}}_{11k}+h^{1^{\ast}}_{22k}=0$ and $H^{1^{\ast}}_{,11}=0$, we obtain
\begin{equation}\label{3.1-15}
H^{2^{\ast}}_{,21}=-H\big(\sum_{k}(h^{1^{\ast}}_{1k})^{2}
+\sum_{k}h^{1^{\ast}}_{1k}h^{2^{\ast}}_{2k}\big).
\end{equation}
According to \eqref{3.1-3} and \eqref{3.1-12}, \eqref{3.1-14} and \eqref{3.1-15} show
\begin{equation}\label{3.1-16}
\lim_{t\rightarrow\infty}H^{2^{\ast}}_{,21}(p_{t})=0, \ \ \lim_{t\rightarrow\infty}h^{1^{\ast}}_{111}(p_{t})=\lim_{t\rightarrow\infty}h^{1^{\ast}}_{221}(p_{t})=0.
\end{equation}
By using \eqref{3.1-1} and \eqref{3.1-16}, we know
\begin{equation*}
\lim_{t\rightarrow\infty}(h^{2^{\ast}}_{222})^{2}(p_{t})
=\lim_{t\rightarrow\infty}(H^{2^{\ast}}_{,2})^{2}(p_{t})= H^{4}.
\end{equation*}
Hence from the second equation of \eqref{3.1-7}, \eqref{3.1-16} and $\bar h^{1^{\ast}}_{11}=\bar h^{2^{\ast}}_{11}=0$, we have that
$$\lim_{t\rightarrow\infty}h^{1^{\ast}}_{112}(p_{t})=0.$$
Namely,
$$\lim_{t\rightarrow\infty}h^{1^{\ast}}_{11k}(p_{t})
=\lim_{t\rightarrow\infty}h^{1^{\ast}}_{22k}(p_{t})=0, \ \ k=1,2.$$
It is easy to draw
$$
\lim_{t\rightarrow\infty}\big(\sum_{i,p}(H^{p^{\ast}}_{,i})^{2}\big)(p_{t})
=\lim_{t\rightarrow\infty}\big(\sum_{i,j,k}(h^{p^{\ast}}_{ijk})^{2}\big)(p_{t})
=\lim_{t\rightarrow\infty}(h^{2^{\ast}}_{222})^{2}(p_{t}).
$$
Then \eqref{3.1-2} implies that
\begin{equation*}
(H^{2}-\bar S)(H^{2}-3\bar S)\leq0.
\end{equation*}
It is a contradiction since $\bar S=3H^{2}$.
The proof of the Proposition \ref{proposition 3.1} is finished.
\end{proof}

\vskip3mm
\noindent
{\it Proof of Theorem \ref{theorem 1.1}}.
From the Proposition \ref{proposition 3.1} and the definition \eqref{1.1-1}
of translators, we show that $T^{\bot}=0$ and $x(M^{2})$ is always tangent $T^{\top}=T$,
which means that $x(M^{2})$ consists of a family
of parallel straight lines. Since $x(M^{2})$ is minimal and complete, we easily obtain that
$x(M^{2})=\mathbb{R}^{2}$.

\vskip3mm
\noindent
{\it Proof of Theorem \ref{theorem 1.2}}.
If $\xi=0$, the theorem reduces to Theorem \ref{theorem 1.1}.
If $\xi\neq0$, we have a globally defined parallel unit normal vector field $\bar e_{1^{\ast}}=\frac{\xi}{|\xi|}$. By rotating $\bar e_{1^{\ast}}$ by an angle of $\frac{\pi}{2}$ in the normal bundle, we obtain another parallel unit normal vector $\bar e_{2^{\ast}}$, which implies that the normal bundle is flat. Since $x$ is Lagrangian,
it follows that the complex structure $J$ is a bundle isometry between the tangent bundle $x_{\ast}(TM^{2})$ and the normal bundle $T^{\perp}M^{2}$. This shows that $M^{2}$ is flat, namely, the Gauss curvature $K\equiv0$. Hence, by $|\vec{H}|^{2}=constant$, we can use the
classification theorem of Hoffman (\cite{Hoff}) to complete the proof of
Theorem \ref{theorem 1.2}.

By applying a research approach similar to the Proposition \ref{proposition 3.1}, we can obtain a similar conclusions about self-expander.

\begin{proposition}\label{proposition 3.2}
Let $x:M^{2}\rightarrow \mathbb{C}^{2}$
be a Lagrangian self-expander with constant squared norm $|\vec{H}|^{2}$ of the mean curvature vector. If the squared norm $S$ of the second fundamental form is bounded from above, then $|\vec{H}|^{2}\equiv0$.
\end{proposition}

\begin{proof}
For $|\vec{H}|^{2}\neq0$, we can always choose a local frame field $\{e_{1}, e_{2}\}$
such that
$$ \vec{H}=H^{1^{\ast}}e_{1^{\ast}}, \ \ H^{1^{\ast}}=|\vec{H}|=H, \ \ H^{2^{\ast}}=h^{2^{\ast}}_{11}+h^{2^{\ast}}_{22}=0.
$$
Hence it is straightforward to see $$H^{2}\leq\frac{4}{3}S.$$
For $S$ being bounded from above, we know that the Bakry-Emery Ricci curvature of $x:M^{2}\rightarrow \mathbb{C}^{2}$ is bounded from below from the lemma \ref{lemma 2.5}. Applying the maximum principle
of Omori-Yau type concerning the operator $\mathcal{L}_{x^{\top}}$ to the function $-S$, we
see that there exists a sequence $\{p_{t}\} \subset M^{2}$ such that
\begin{equation*}
\lim_{t\rightarrow\infty} S(p_{t})=\inf S, \ \
\lim_{t\rightarrow\infty} |\nabla S|(p_{t})=0, \ \
\lim_{t\rightarrow\infty}\mathcal{L}_{x^{\top}} S(p_{t})\geq 0.
\end{equation*}
Since $|\vec{H}|^{2}=constant$, the lemma \ref{lemma 2.3} gives
\begin{equation}\label{3.1-17}
\sum_{i,p}(H^{p^{\ast}}_{,i})^{2}
=H^{2}+\sum_{i,j,p,q}H^{p^{\ast}}h^{p^{\ast}}_{ij}H^{q^{\ast}}h^{q^{\ast}}_{ij}
\end{equation}
and
\begin{equation}\label{3.1-18}
\frac{1}{2}\mathcal{L}_{x^{\top}} S=\sum_{i,j,k}(h^{p^{\ast}}_{ijk})^{2}
-\sum_{i,p}(H^{p^{\ast}}_{,i})^{2}-\frac{1}{2}(H^{2}-S)(H^{2}-3S-2).
\end{equation}
Thus, one can assume
\begin{equation*}
\lim_{t\rightarrow\infty}S(p_{t})=\inf S=\bar S, \ \ \lim_{t\rightarrow\infty}h^{p^{\ast}}_{ij}(p_{t})=\bar h^{p^{\ast}}_{ij}, \ \
\lim_{t\rightarrow\infty}H^{p^{\ast}}_{,i}(p_{t})=\bar H^{p^{\ast}}_{,i}, \ \ i, j, p=1, 2
\end{equation*}
since $S$ is bounded from above.
Without loss of the generality, we can
assume $h^{p^{\ast}}_{ij}(p_{t})\neq0$ for $i,j,p=1,2,3$. Unless otherwise specified, the following equations are considered at point $p_{t}\in M^{2}$.

What we want to do is prove $\inf S=0$. In the following text, we assume that $\inf S>0$.
Using the fact that $|\vec{H}|^{2}=constant$  and $|\nabla H^{2}|^{2}=4\sum_{k}(\sum_{p}H^{p^{\ast}}H^{p^{\ast}}_{,k})^{2}$, we obtain
we get that
\begin{equation}\label{3.1-19}
H^{1^{\ast}}_{,k}=H^{k^{\ast}}_{,1}=0, \ \ h^{1^{\ast}}_{11k}+h^{1^{\ast}}_{22k}=0, \ \ k=1, 2.
\end{equation}
By using the first formula of \eqref{2.1-12} and $h^{2^{\ast}}_{11}+h^{2^{\ast}}_{22}=0$, we obtain that
\begin{equation}\label{3.1-20}
H^{1^{\ast}}_{,1}=-h^{1^{\ast}}_{11}\langle x, e_{1} \rangle-h^{2^{\ast}}_{11}\langle x, e_{2} \rangle, \ \
H^{1^{\ast}}_{,2}=-h^{2^{\ast}}_{11}\langle x, e_{1} \rangle-h^{1^{\ast}}_{22}\langle x, e_{2} \rangle
\end{equation}
and
\begin{equation}\label{3.1-21}
H^{2^{\ast}}_{,1}=-h^{2^{\ast}}_{11}\langle x, e_{1} \rangle-h^{1^{\ast}}_{22}\langle x, e_{2} \rangle, \ \
H^{2^{\ast}}_{,2}=-h^{1^{\ast}}_{22}\langle x, e_{1} \rangle+h^{2^{\ast}}_{11}\langle x, e_{2} \rangle.
\end{equation}
Choosing $\nabla_{k}S=2a_{k}$ for $k=1,2$, \eqref{3.1-17}, \eqref{3.1-19}, \eqref{3.1-20} and $\lim_{t\rightarrow\infty} |\nabla S|(p_{t})=0$ can be written as

\begin{equation}\label{3.1-22}
(H^{2^{\ast}}_{,2})^{2}=H^{2}\big(1+\sum_{i,j}(h^{1^{\ast}}_{ij})^{2}\big),
\end{equation}

\begin{equation}\label{3.1-23}
\big(h^{1^{\ast}}_{11}h^{1^{\ast}}_{22}-( h^{2^{\ast}}_{11})^{2}\big)\langle x, e_{1} \rangle=0, \ \
\big(h^{1^{\ast}}_{11}h^{1^{\ast}}_{22}-( h^{2^{\ast}}_{11})^{2}\big)\langle x, e_{2} \rangle=0
\end{equation}
and
\begin{equation}\label{3.1-24}
\begin{aligned}
&\big((h^{1^{\ast}}_{11}-3h^{1^{\ast}}_{22})^{2}+12(h^{2^{\ast}}_{11})^{2}\big)h^{1^{\ast}}_{111}
+4(h^{2^{\ast}}_{11})^{2} h^{2^{\ast}}_{222}=b_{1}, \\
&\big((h^{1^{\ast}}_{11}-3h^{1^{\ast}}_{22})^{2}+12(h^{2^{\ast}}_{11})^{2}\big) h^{2^{\ast}}_{111}-h^{2^{\ast}}_{11}(h^{1^{\ast}}_{11}-3h^{1^{\ast}}_{22})h^{2^{\ast}}_{222}
=b_{2}.
\end{aligned}
\end{equation}
where $b_{1}=(h^{1^{\ast}}_{11}-3h^{1^{\ast}}_{22})a_{1}-4h^{2^{\ast}}_{11} a_{2}$,
$b_{2}=3h^{2^{\ast}}_{11} a_{1}+(h^{1^{\ast}}_{11}-3h^{1^{\ast}}_{22})a_{2}$ and
$\lim_{t\rightarrow\infty}b_{k}(p_{t})=0$ for $k=1,2$.

Using \eqref{3.1-17}--\eqref{3.1-24}, by making use of the same proof as in the proof of the Proposition \ref{proposition 3.1}, we can also obtain
\begin{equation}\label{3.1-25}
\bar h^{2^{\ast}}_{11}=0.
\end{equation}

\noindent
From  $|\vec{H}|^{2}=constant$ and \eqref{3.1-19}, we naturally obtain
\begin{equation}\label{3.1-26}
\sum_{p}H^{p^{\ast}}_{,1}H^{p^{\ast}}_{,1}
+H^{1^{\ast}}H^{1^{\ast}}_{,11}=0, \ \ H^{1^{\ast}}_{,11}=0.
\end{equation}
Besides, \eqref{3.1-17} yields that
\begin{equation*}
\sum_{i,p}H^{p^{\ast}}_{,i}H^{p^{\ast}}_{,ik}
=\sum_{i,j}(\sum_{p}H^{p^{\ast}}h^{p^{\ast}}_{ij})(\sum_{p}H^{p^{\ast}}_{,k}h^{p^{\ast}}_{ij}
+\sum_{p}H^{p^{\ast}}h^{p^{\ast}}_{ijk})
\end{equation*}
Hence by \eqref{3.1-19}, we obtain
\begin{equation}\label{3.1-27}
H^{2^{\ast}}_{,2}H^{2^{\ast}}_{,21}
=H^{2}\sum_{ij}h^{1^{\ast}}_{ij}h^{1^{\ast}}_{ij1}
=H^{2}\big((h^{1^{\ast}}_{11}-h^{1^{\ast}}_{22})h^{1^{\ast}}_{111}+2 h^{2^{\ast}}_{11}h^{2^{\ast}}_{111}\big)
\end{equation}
For $\nabla_{k}S=2a_{k}$ for $k=1,2$, by \eqref{3.1-19}, we know
\begin{equation}\label{3.1-28}
(h^{1^{\ast}}_{11}-3 h^{1^{\ast}}_{22})h^{1^{\ast}}_{111}+4 h^{2^{\ast}}_{11}h^{2^{\ast}}_{111}=a_{1}.
\end{equation}
Combining \eqref{3.1-27}, \eqref{3.1-28} and $H=h^{1^{\ast}}_{11}+h^{1^{\ast}}_{22}$, we concludes that
\begin{equation}\label{3.1-29}
H^{2^{\ast}}_{,2}H^{2^{\ast}}_{,21}
=\frac{1}{2}H^{2}(Hh^{1^{\ast}}_{111}+a_{1}).
\end{equation}
It follows from \eqref{2.1-12} that

\begin{equation*}
\begin{aligned}
&H^{1^{\ast}}_{,11}=-\sum_{k}h^{1^{\ast}}_{11k}\langle x, e_{k}\rangle-h^{1^{\ast}}_{11}-H\sum_{k}(h^{1^{\ast}}_{1k})^{2}, \\
&H^{2^{\ast}}_{,21}=-\sum_{k}h^{2^{\ast}}_{21k}
\langle x, e_{k}\rangle-h^{1^{\ast}}_{22}-H\sum_{k}h^{1^{\ast}}_{1k}h^{2^{\ast}}_{2k}.
\end{aligned}
\end{equation*}
Then by $H^{1^{\ast}}_{,11}=0$, we obtain
\begin{equation}\label{3.1-30}
H^{2^{\ast}}_{,21}=-H\big(1+\sum_{k}(h^{1^{\ast}}_{1k})^{2}
+\sum_{k}h^{1^{\ast}}_{1k}h^{2^{\ast}}_{2k}\big).
\end{equation}

Next, we will use $\bar h^{2^{\ast}}_{11}=0$ to complete the proof of the Proposition \ref{proposition 3.2}.

If $\bar h^{1^{\ast}}_{11}\bar h^{1^{\ast}}_{22}\neq0$, using \eqref{3.1-21} and \eqref{3.1-23}, we have
\begin{equation*}
\lim_{t\rightarrow\infty}\langle x, e_{1} \rangle(p_{t})
=\lim_{t\rightarrow\infty}\langle x, e_{2} \rangle(p_{t})=0, \ \ \lim_{t\rightarrow\infty}\big(\sum_{i,p}(H^{p^{\ast}}_{,i})^{2}\big)(p_{t})
=\lim_{t\rightarrow\infty}\big(H^{2^{\ast}}_{,2})^{2}(p_{t})=0.
\end{equation*}
Then \eqref{3.1-22} yields that
\begin{equation*}
H^{2}\big(1+\sum_{i,j}(\bar h^{1^{\ast}}_{ij})^{2}\big)=0
\end{equation*}
It contradicts the hypothesis.

\noindent
If $\bar h^{1^{\ast}}_{11}=0$, it is
easy to see that $\bar h^{1^{\ast}}_{22}\neq0$ since $|\vec{H}|^{2}\neq0$.
It follows from \eqref{3.1-19}, \eqref{3.1-22}, \eqref{3.1-29} and \eqref{3.1-30} that
\begin{equation}\label{3.1-31}
\lim_{t\rightarrow\infty}H^{2^{\ast}}_{,21}(p_{t})=-H, \ \ \big(\lim_{t\rightarrow\infty}h^{1^{\ast}}_{111}(p_{t})\big)^{2}
=\big(\lim_{t\rightarrow\infty}h^{1^{\ast}}_{221}(p_{t})\big)^{2}=\frac{4(1+H^{2})}{H^{2}},
\end{equation}
which implies
\begin{equation}\label{3.1-32}
|\lim_{t\rightarrow\infty}h^{2^{\ast}}_{222}(p_{t})|<\infty.
\end{equation}
By use of the first equation of \eqref{3.1-24}, \eqref{3.1-32} and $\bar h^{1^{\ast}}_{11}=\bar h^{2^{\ast}}_{11}=0$, we know
$$\lim_{t\rightarrow\infty}h^{1^{\ast}}_{111}(p_{t})=0,$$
which is in contradiction to \eqref{3.1-31}.

\noindent
If $\bar h^{1^{\ast}}_{22}=0$, one sees $\bar h^{1^{\ast}}_{11}\neq0$ since $|\vec{H}|^{2}\neq0$. \eqref{3.1-30} yields that
\begin{equation}\label{3.1-33}
\lim_{t\rightarrow\infty}H^{2^{\ast}}_{,21}(p_{t})=-H(1+H^{2}).
\end{equation}
It follows from \eqref{3.1-19}, \eqref{3.1-22}, \eqref{3.1-29} and \eqref{3.1-33} that
\begin{equation*}
\big(\lim_{t\rightarrow\infty}h^{1^{\ast}}_{111}(p_{t})\big)^{2}
=\big(\lim_{t\rightarrow\infty}h^{1^{\ast}}_{221}(p_{t})\big)^{2}
=\frac{4(1+H^{2})^{3}}{H^{2}}, \ \ |\lim_{t\rightarrow\infty}h^{2^{\ast}}_{222}(p_{t})|<\infty.
\end{equation*}
Then by the first equation of \eqref{3.1-24} and $\bar h^{1^{\ast}}_{22}=\bar h^{2^{\ast}}_{11}=0$, we know
$$\lim_{t\rightarrow\infty}h^{1^{\ast}}_{111}(p_{t})=0.$$
This also yields the contradiction. The proof of the Proposition \ref{proposition 3.2} is finished.

\end{proof}

\vskip3mm
\noindent
{\it Proof of Theorem \ref{theorem 1.3}}.
By the Proposition \ref{proposition 3.2}, we have $|\vec{H}|^{2}\equiv0$. Thus,
it follows from the definition of lagrangian self-expander and the first equation of \eqref{2.1-12} that
$$H^{p^{\ast}}=\langle x, e_{p^{\ast}}\rangle=0, \ \ \sum_{k}h^{p^{\ast}}_{ik}\langle x, e_{k}\rangle=0, \ \ i, p=1, 2.$$
That is,
$$h^{1^{\ast}}_{11}+h^{1^{\ast}}_{22}=0, \ \ h^{2^{\ast}}_{11}+h^{2^{\ast}}_{22}=0$$
and
\begin{equation*}
h^{1^{\ast}}_{11}\langle x, e_{1}\rangle+h^{1^{\ast}}_{12}\langle x, e_{2}\rangle=0, \ \
h^{1^{\ast}}_{21}\langle x, e_{1}\rangle+h^{1^{\ast}}_{22}\langle x, e_{2}\rangle=0.
\end{equation*}
Then by the symmetry of indices, we infer
$$\big((h^{1^{\ast}}_{11})^{2}+(h^{2^{\ast}}_{11})^{2}\big)\langle x, e_{k}\rangle=0, \ \ k=1,2.$$
Let us suppose $(h^{1^{\ast}}_{11})^{2}+(h^{2^{\ast}}_{11})^{2}\neq0$, we obtain
$$0=\langle x, e_{k}\rangle_{,k}=1+\sum_{p}h^{p^{\ast}}_{kk}\langle x, e_{p^{\ast}}\rangle=1.$$
It is impossible. Thus, $ h^{p^{\ast}}_{ik}=0 $ for $i,k,p=1,2$.
Clearly show that $x(M^{2})=\mathbb{R}^{2}$ through the origin.

\begin{flushright}
$\square$
\end{flushright}

\noindent{\bf Acknowledgements.}
The first author was partially supported by Natural Science Foundation of Henan Province Grant No.242300421686. The second author was partly supported by grant No.12171164 of NSFC, GDUPS (2018), Guangdong Natural Science Foundation Grant No.2023A1515010510.

\end{document}